\documentclass[12pt]{amsart}

\usepackage{lineno,hyperref}
\modulolinenumbers[5]

\usepackage[utf8]{inputenc}
\usepackage{lmodern}
\usepackage{algorithm2e}

\usepackage{url}
\usepackage{amsmath, amsfonts, amsbsy, mathrsfs, amsthm, amssymb, amscd, accents, bm, color}
\usepackage{mathtools}
\usepackage{mdwlist}
\usepackage{paralist}
\usepackage{graphicx}
\usepackage{epstopdf}

\newtheorem{definition}{Definition}[section]
\newtheorem{theorem}[definition]{Theorem}
\newtheorem{lemma}[definition]{Lemma}
\newtheorem{corollary}[definition]{Corollary}
\newtheorem{example}[definition]{Example}
\newtheorem{proposition}[definition]{Proposition}
\newtheorem{algr}[definition]{Algorithm}

\def\N{{\mathbb N}}
\def\Z{{\mathbb Z}}
\def\R{{\mathbb R}}

\def\A{{\mathcal A}}

\def\1{\textbf{1}}

\DeclareMathOperator{\graph}{graph}

\begin{document}

\title{On the Mathematical Validity of the Higuchi Method}

\author{Lukas Liehr\\Peter Massopust${}^*$}
\thanks{${}^*$Corresponding author}
\address{Centre of Mathematics\\Technical University of Munich\\Boltzmannstrasse 3\\85748 Garching b. Munich\\ Germany}
\email{massopust@ma.tum.de}

%
%

\begin{abstract}
In this paper, we discuss the Higuchi algorithm which serves as a widely used estimator for the box-counting dimension of the graph of a bounded function $f : [0,1] \to \R$. We formulate the method in a mathematically precise way and show that it yields the correct dimension for a class of non-fractal functions. Furthermore, it will be shown that the algorithm follows a geometrical approach and therefore gives a reasonable estimate of the fractal dimension of a fractal function. We conclude the paper by discussing the robustness of the method and show that it can be highly unstable under perturbations.

\end{abstract}


\maketitle

\section{Introduction and Preliminaries}

In the following, we provide a mathematical investigation of the Higuchi method \cite{higuchi}, an algorithm which aims at approximating the box-counting dimension of the graph $\Gamma$ of a bounded real valued function $f: [0,1] \to \R$ arising from a (measured) time series. For numerous applications, the graphs of these functions exhibit fractal characteristics and the Higuchi method is therefore used in various areas of science where such functions appear. It has been applied to such diverse subjects  as analyzing heart rate variability \cite{GKKTSK}, characterizing primary waves in seismograms \cite{GMPBA}, analyzing changes in the electroencephalogram in Alzheimer’s disease \cite{HJSI}, and digital images \cite{ahammer}. In the original article, the validity of the algorithm was based on a number of numerical simulations and not on an exact mathematical validation. 

This article presents mathematically precise conditions of when and for what type of functions the Higuchi method gives the correct value for the box-counting dimension of its graph. Of particular importance in this context are also the robustness and the behavior of the algorithm under perturbation of the data. The latter reflects imprecise measurements and the existence of measurement errors. We show that the Higuchi method as proposed in \cite{higuchi} is neither robust nor stable under perturbations of the data.

This paper uses the following notation and terminology. The collection of all continuous real-valued functions $f: [a,b] \to \R$ will be denoted by $C[a,b]$. Such functions $f$ will be regarded as models for continuous signals measured over time $t \in [a,b]$ with values $f(t)\in \R$. A function $f: [a,b] \to \R$ is called bounded if there exists a nonempty interval $[m,M]\subset \R$ such that $f(x) \in [m,M]$, for all $x\in [a,b]$.

In many applications one is interested in characterizing the irregularity of the graph $\Gamma$ of a bounded function $f$:
\begin{equation*}
\Gamma \coloneqq \mathrm{graph}(f) \coloneqq \{ (t,f(t)) \ | \ t \in [a,b] \} \subset \R^2.
\end{equation*}
 
A widely used measurement for the irregularity of $f$ is given by the box-counting dimension of its graph $\Gamma$. 

\begin{definition}[Box-counting dimension]
Suppose $F$ be a nonempty bounded subset of $\R^n$ and $\delta > 0$. Denote by $N_\delta(F)$ the smallest number of boxes of side length equal to $\delta$ needed to cover $F$. The box-counting dimension of $F$ is defined to be
\begin{equation}\label{defboxdim}
\mathrm{dim}_B(F) = \lim_{\delta \to 0+} \frac{\log(N_\delta(F))}{\log(\frac{1}{\delta})},
\end{equation}
provided the limit exists.
\end{definition}
\noindent
For more details, we refer the interested reader to, for instance, \cite{falconer1,falconer2, massopust}. It follows from the above definition that any nonempty bounded subset $F$ of $\R^n$ satisfies $\mathrm{dim}_B(F) \leq n$. 

We note that if $f \in C[a,b]$, the graph of $f$ is bounded and thus $1 \leq \mathrm{dim}_B(\Gamma) \leq 2$. The case where $\mathrm{dim}_B(\Gamma) = 2$, i.e., when $f$ is a so-called space-filling or Peano curve will not be considered here.

\section{The Higuchi Fractal Dimension}
In this section, we review the definition of the Higuchi method as introduced in \cite{higuchi} and derive some immediate consequences. Without loss of generalization, we restrict ourselves to the unit interval $[0,1]$.

For a given bounded function $f: [0,1] \to \R$ and $N \in \N, N \geq 2$, we define a finite time series by
\begin{equation*}
X_N : \{ 1, \dots, N\} \to \R, \quad X_N(j) := f\left( \frac{j-1}{N-1} \right ).
\end{equation*}
The time series $X_N$ represents $N$ samples of $f$ obtained by a uniform partitioning of $[0,1]$ in $N-1$ subintervals.

The Higuchi method uses the time series $X_N$ and a parameter $k_\mathrm{max} \in \N$ with the property $1 \leq k_\mathrm{max} \leq \lceil\frac{N}{2}\rceil$ (here, $\lceil \cdot \rceil : \R \to \Z$ denotes the ceiling function) to compute a fractal dimension, which we refer to as the Higuchi fractal dimension (HFD). This approach is originally stated in \cite{higuchi}. We start by formulating the method in a mathematically precise way. 

For this purpose, we first present an algorithmic description of the Higuchi method based on the original article \cite{higuchi}.

\begin{algr}[Higuchi fractal dimension]\label{algo}
\textcolor{white}{.}\newline
\begin{algorithm}[H]

\SetAlgoLined
 \textbf{Input}: $2 \leq N\in \N, \ X_N : \{ 1, \dots, N\} \to \R$, $k_\mathrm{max} \in \N$, $k_\mathrm{max} \leq \lceil \frac{N}{2} \rceil$ \;
 \textbf{Output}: Higuchi fractal dimension D = $\verb|HFD|(X_N,N,k_\mathrm{max})$ \;
 
 \For{$k = 1, \dots, k_\mathrm{max} $}{
 \For{$m=1, \dots, k$}{
  $C_{N,k,m} = \frac{N-1}{\left\lceil \frac{N-m}{k} \right \rceil k}$\;
  $V_{N,k,m} = \sum\limits_{i=1}^{\left \lceil \frac{N-m}{k} \right \rceil} |X_N(m+ik)-X_N(m+(i-1)k)|$\;
  $L_m(k) = \frac{1}{k} C_{N,k,m}V_{N,k,m}$\;
  }
  $L(k) = \frac{1}{k} \sum\limits_{m=1}^k L_m(k)$\;
 }
 $\mathcal{I} = \{ k \in \{ 1, \dots, k_\mathrm{max} \} \ | \ L(k) \neq 0 \}$\;
 $\mathcal{Z} = \left \{ \left ( \log \frac{1}{k} ,\log L(k)  \right ) \ | \ k \in I \right \}$\;
 \eIf{$|\mathcal{I}| \leq 1$}{
 D=1\;
 }{
 $f(x) = Dx+h$ :  best fitting affine function through $\mathcal{Z}$\;
 }
\end{algorithm}

\end{algr}

\begin{figure}[h!]
\centering
\includegraphics[scale=0.53]{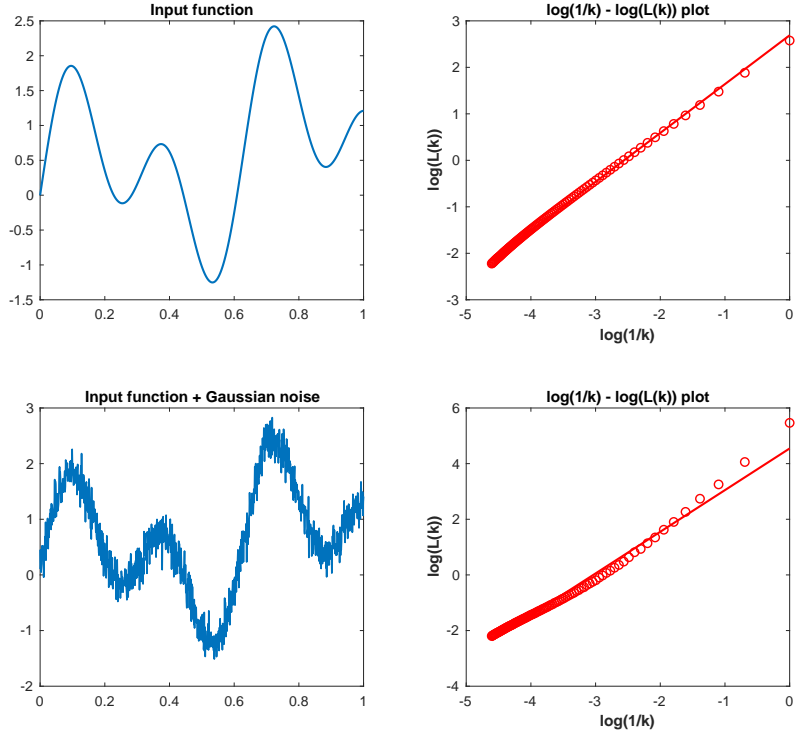}
\caption{Two input functions on the left with their corresponding regression line through the data set $\mathcal{Z}$ on the right. The slope of the first line is $\approx 1.05$ and the slope of the second is $\approx 1.49$.}\label{fig1}
\end{figure}

The Higuchi algorithm uses $N$ samples of $f$ and then forms $k_\mathrm{max}$ ``lengths" $L(1)$, $\dots$, $L(k_\mathrm{max})$. Each length $L(k)$ is the average over $k$ values $L_1(k), \dots, L_k(k)$,
\begin{equation*}
L(k) = \frac{1}{k} \sum_{m=1}^k L_m(k) \ ,
\end{equation*}
and the values $L_m(k)$ are computed via
\begin{equation*}
L_m(k) = \frac{1}{k} \frac{N-1}{\left \lceil \frac{N-m}{k} \right \rceil k}  {\sum_{i=1}^{\left \lceil \frac{N-m}{k} \right \rceil} \left|X_N(m+ik)-X_N(m+(i-1)k)\right|}.
\end{equation*}
For notational simplicity and the mathematical analysis of the terms appearing in $L_m(k)$, we set
\[
C_{N,k,m} := \frac{N-1}{\left \lceil \frac{N-m}{k} \right \rceil k}
\]
and
\[
V_{N,k,m} := \sum_{i=1}^{\left \lceil \frac{N-m}{k} \right \rceil} \left|X_N(m+ik)-X_N(m+(i-1)k)\right|.
\]
We use the notation $V_{N,k,m}$ to emphasize that this value can be regarded as an approximation for the total variation of $f$ on the interval $[0,1]$. The relation between the total variation of $f$ and the Higuchi method will be discussed below in Section 3.

The algorithm then collects all indices $k$ with $L(k) \neq 0$ into an index set $\mathcal{I} = \{ k \in \{ 1, \dots, k_\mathrm{max} \} \ | \ L(k) \neq 0 \}$ and defines a data set 
\begin{equation*}
\mathcal{Z} = \left \{ \left ( \log \frac{1}{k} ,\log L(k)  \right ) \ | \ k \in \mathcal{I} \right \}.
\end{equation*}

The slope $D$ of the best fitting affine function through $\mathcal{Z}$ (in the least-square sense) is defined to be the Higuchi fractal dimension of $f$ with parameters $N$ and $k_\mathrm{max}$: $D = \verb|HFD|(X_N,N,k_\mathrm{max})$.

If the index set $\mathcal{I}$ is empty or contains just one element then $D$ is defined to be $1$. We emphasize that in general we cannot guarantee that the values $L(k)$ differ from zero (the simplest examples are constant functions where $L(k) = 0$ for all $k$). Therefore, the definition of the index set $\mathcal{I}$ is necessary.

Recall that for a given data set $\{(x_i,y_i) \ | \ i = 1, \ldots, n\}$ the slope $D$ of an affine function that is the best least squares fit of the data can be computed via
\begin{equation}\label{slope}
D = \frac{\sum\limits_{i=1}^n (x_i - \overline{x})(y_i-\overline{y})}{\sum\limits_{i=1}^n (x_i-\overline{x})^2},
\end{equation}
where $\overline{x}$ and $\overline{y}$ are the mean values of $\{x_i \ | \ i = 1, \ldots, n\}$ and $\{y_i \ | \ i = 1, \ldots, n\}$, respectively. For every $N$ samples of $f$ one has to specify a $k_\mathrm{max}$. Moreover, we need $1 \leq k_\mathrm{max} \leq \lceil\frac{N}{2} \rceil$ since otherwise the values $V_{N,k,m}$ are not well defined for all $1 \leq k \leq k_\mathrm{max}$ and all $1 \leq m \leq k$.

\begin{definition}[Admissible Input]
We call a pair $(N,k_\mathrm{max}) \in \N^2$ admissible  for the Higuchi method if
\begin{equation*}
N \geq 2 \ \ \mathrm{and} \ \  k_\mathrm{max} \leq \left\lceil \tfrac{N}{2} \right\rceil.
\end{equation*}
\end{definition}

The notation $\verb|HFD| (f,N,k_{\mathrm{max}})$ is used to denote the Higuchi fractal dimension of $f$ with parameters $N$ and $k_\mathrm{max}$.
The following observations follow immediately from the definition of the algorithm.

\begin{proposition}\label{basics}
Let $f : [0,1] \to \R$ be a bounded function and $(N,k_\mathrm{max}) \in \N^2$ an admissible pair.
\begin{enumerate}
  \item If f is constant then $L(k)=0$ for all $k \in \{ 1, \dots, k_\mathrm{max} \}$ and therefore $\verb|HFD| (f,N,k_{\mathrm{max}}) = 1$.
  \item If  $L(k) \propto k^{-D}$ for all $k \in \mathcal{I}$ with $|\mathcal{I}| \geq 2$, then $\verb|HFD| (f,N,k_{\mathrm{max}}) = D$.
  \item If $f$ is affine then $\verb|HFD| (f,N,k_{\mathrm{max}}) = 1$ for every admissible choice of $(N,k_\mathrm{max}) \in \N^2$, i.e.,
  $$\verb|HFD| (f,N,k_{\mathrm{max}}) = \mathrm{dim}_B(\Gamma).$$
  \end{enumerate}
\end{proposition}

\begin{proof}
(1) If $f(x)=c$ for all $x \in [0,1]$ and some $c \in \R$ then $V_{N,k,m} = 0,$ for all $k \in \{ 1, \dots, k_\mathrm{max} \}$ and for all $m \in \{ 1, \dots, k \}$. Therefore, $\frac{1}{k}\sum\limits_{m=1}^k L_m(k) = 0 = L(k)$. Hence, $\mathcal{I} = \emptyset$ which implies $\verb|HFD| (f,N,k_{\mathrm{max}}) = 1$.

(2) Assume $L(k) \propto k^{-D}$ for $k \in \mathcal{I}$. Then there exists a $c > 0$ such that $L(k) = ck^{-D}$ for all $k \in \mathcal{I}$. Since $|\mathcal{I}| \geq 2$, there exist $k_1, k_2 \in \mathcal{I}$ with $k_1 \neq k_2$. Without loss of generality we may assume that $k_1 < k_2$. It follows that
\begin{equation*}
\frac{\log(L(k_2)) - \log(L(k_1))}{\log(\frac{1}{k_2}) - \log (\frac{1}{k_1})} = D.
\end{equation*}
Hence, any line through two consecutive points in $\mathcal{Z}$ has slope $D$. Therefore, a linear regression line through all points in $\mathcal{Z}$ has slope $D$ as well. All in all, we obtain $\verb|HFD| (f,N,k_{\mathrm{max}}) = D$.

(3) If $f$ is constant then $\mathcal{I} = \emptyset$ and therefore $\verb|HFD| (f,N,k_{\mathrm{max}}) = 1$. If $f$ is not constant then there exists an $a \neq 0$ and a $b \in \R$ such that $f(x) = ax+b$. Consequently,
\begin{equation*}
V_{N,k,m} = \sum_{i=1}^{\left \lceil \frac{N-m}{k} \right \rceil} |a| \left | \frac{m+ik-1}{N-1} - \frac{m+(i-1)k-1}{N-1} \right | = \frac{|a|\left \lceil \frac{N-m}{k} \right \rceil}{N-1} k
\end{equation*}
which implies
\begin{equation*}
L_m(k) = \frac{|a|}{k} = L(k).
\end{equation*}
Therefore, $L(k) \propto k^{-1}$ and the assertion follows from \eqref{defboxdim}.
\end{proof}

The third part of Proposition \ref{basics} shows that the normalization factor $C_{N,k,m}$ cannot be dropped without modifications since otherwise the computed dimension of an affine function could already differ from 1. But for an affine function the graph is a straight line and therefore the simplest example of a geometric object with box-counting dimension 1. We discuss the geometric meaning of the constants $C_{N,k,m}$ in more detail in Section 4.

\section{Relation to Functions of Bounded Variation}

In the original paper by Higuchi \cite{higuchi} and many other papers where the Higuchi method is applied to data sets, the values $L_m(k)$ are identified as the normalized lengths of a curve generated by the sub-time series
\begin{equation*}
X(m), X(m+k), X(m+2k), \dots, X \left ( m+\left\lceil \frac{N-m}{k} \right\rceil k \right ),
\end{equation*}
with $m \in \{1, \ldots, k\}$. In fact, the sum
\begin{equation*}
\sum_{i=1}^{\left\lceil \frac{N-m}{k} \right\rceil} \left|X_N(m+ik)-X_N(m+(i-1)k)\right|
\end{equation*}
does \emph{not} measure the Euclidean length of the curve $t \mapsto (t,f(t))$ but we can regard it as an \emph{approximation} for the total variation of $f$. Let us recall the definition of the total variation of a function $f : [a,b] \to \R$.

\begin{definition}[Total variation]
Let $f : [a,b] \to \R$ be a function and $a,b\in \R$ with $a<b$. For a partition $P=\{t_0, \dots, t_N \}$ of $[a,b]$ with $a = t_0 < t_1 < \cdots < t_N = b$, define
\begin{equation*}
V_P(f) \coloneqq \sum_{i=1}^N |f(t_i)-f(t_{i-1})|
\end{equation*}
and
\begin{equation*}
|P| \coloneqq \max_{i = 1, \dots, N} |t_i - t_{i-1}|.
\end{equation*}
We call $|P|$ the mesh size of $P$. Further set
\begin{equation*}
V_a^b(f) \coloneqq \sup_P V_P(f)
\end{equation*}
where the supremum is taken over all partitions of $[a,b]$. Then $V_a^b(f)$ is called the total variation of $f$.
If $V_a^b(f) < \infty$ then $f$ is said to have bounded variation on $[a,b]$ and we define
\begin{equation*}
BV[a,b] \coloneqq \{ f : [a,b] \to \R \ | \ V_a^b(f) < \infty \}.
\end{equation*}
\end{definition}

As an input for the Higuchi algorithm one usually uses non-constant functions since otherwise $L(k) = 0$ for all $k$ and therefore there is no non-horizontal regression line through the data set $\mathcal{Z}$. We can express this condition in terms of the total variation of $f$.

\begin{lemma}\label{lem3.2}
For a function $f : [a,b] \to \R$, $V_a^b(f) = 0$ if and only if $f$ is constant on $[a,b]$.
\end{lemma}
\begin{proof}
The statements follow directly from the definition of $V_a^b(f)$.
\end{proof}
It is known that if $f \in BV[a,b]$ then $\mathrm{dim}_B(\mathrm{graph}(f)) = 1$. See, for instance,  \cite{falconer1} for a proof. 

In the following, we examine whether or not the Higuchi algorithm yields the correct dimension of graph($f$) if $f$ is continuous and of bounded variation. To do so we first show a relation between the values $V_{N,k,m}$ defined in Algorithm \ref{algo} and the variation of $f$. This relation is based on the following result.

\begin{theorem}\label{thm3.3}
Let $f \in C[a,b] \cap BV[a,b]$ and let $(P_n)_{n \in \N}$ be a sequence of partitions of $[a,b]$ such that $|P_n| \to 0$ as $n \to \infty$. Then
\begin{equation*}
\lim_{n \to \infty} V_{P_n}(f) = V_a^b(f).
\end{equation*}
\end{theorem}
\begin{proof}
See the Appendix.
\end{proof}

\begin{proposition}\label{prop3.4}
Let $f \in C[0,1] \cap BV[0,1]$ and $V_{N,k,m}$ be defined as in the Higuchi algorithm (with $m,k \in \N$ fixed). Then
\begin{equation*}
\lim_{N \to \infty} V_{N,k,m} = V_0^1(f).
\end{equation*}
\end{proposition}
\begin{proof}
Define partitions
\begin{equation*}
P_N \coloneqq  \left\{ \frac{m+ik-1}{N-1} \;\bigg\vert\; i = 1, \dots, \left\lceil\frac{N-m}{k} \right\rceil \right\}  \cup \{0,1\}.
\end{equation*}
Since $|P_N| = \mathcal{O}(\frac{1}{N})$ it follows from Theorem \ref{thm3.3} that
\begin{equation*}
\lim_{N \to \infty} V_{P_N}(f) = V_0^1(f).
\end{equation*}
Define $l_N := \frac{m-1}{N-1}$ and $r_N := \frac{m+\left\lceil\frac{N-m}{k} \right\rceil k-1}{N-1}$. Then
\begin{equation*}
V_{P_N}(f) = V_{N,k,m} + {|f(0)-f(l_N)| + |f(1)-f(r_N)|}.
\end{equation*}
Set $e_N:= |f(0)-f(l_N)| + |f(1)-f(r_N)|$.  Since $l_N \to 0$ and $r_N \to 1$ as $N \to \infty$, we have  $e_N \to 0$ by the continuity of $f$. Consequently,
\begin{equation*}
\lim_{N \to \infty} V_{P_N}(f) = \lim_{N \to \infty} V_{N,k,m} = V_0^1(f).\qedhere
\end{equation*}
\end{proof}

The preceding proposition yields the next result.

\begin{theorem}
Let $f \in C[0,1] \cap BV[0,1]$. Then
\begin{equation*}
\verb|HFD|(f,N,2) \to 1 = \mathrm{dim}_B(\mathrm{graph}(f))\quad\text{as $N\to\infty$}.
\end{equation*}
In other words, $\verb|HFD|(f,N,k_\mathrm{max})$ converges to $\mathrm{dim}_B(\mathrm{graph}(f))$ if we choose an admissible sequence $(N,k_\mathrm{max}) \subset \N^2$ in such a way that $k_\mathrm{max}=2$ is fixed and $N$ goes to infinity.
\end{theorem}
\begin{proof}
Since $k_\mathrm{max}=2$, we only have to consider the values $L(1)$ and $L(2)$. We have
\begin{equation*}
L(1) = V_{N,1,1} \ \ \mathrm{and} \ \ L(2) = \tfrac{1}{4}(C_{N,2,1}V_{N,2,1} + C_{N,2,2}V_{N,2,2}).
\end{equation*}
If $f$ is constant then the result follows from Proposition \ref{basics}. If $f$ is not constant then $V_0^1(f) > 0$ by Lemma \ref{lem3.2}. Furthermore, we have $\lim\limits_{N\to\infty} V_{N,k,m} = V_0^1(f) > 0$ by Proposition \ref{prop3.4} and $\lim\limits_{N\to\infty} C_{N,k,m} = 1$. Therefore, there exists an $\overline{N}\in\N$ such that 
\begin{equation*}
L(1),L(2) > 0, \quad \forall N \geq \overline{N}.
\end{equation*}
The slope $d_N$ of a linear regression line through
\begin{equation*}
\mathcal{Z} = \left\{ (\log 1, \log L(1)), (\log \tfrac{1}{2}, \log L(2)) \right\}
\end{equation*} 
is given by
\begin{equation*}
d_N = \frac{\log L(2) - \log L(1)}{\log \frac{1}{2} - \log 1} = 1 + {\frac{\log (\frac{1}{2}(C_{N,2,1}V_{N,2,1} + C_{N,2,2}V_{N,2,2}) - \log V_{N,1,1}}{\log \frac{1}{2}}}.
\end{equation*}
The fraction on the right-hand side goes to zero as $N\to\infty$. But $d_N$ is exactly $\verb|HFD|(f,N,2)$ and the result follows. 
\end{proof}

\begin{example}
The function $f_c(x) = x^2\sin(\frac{c}{x}), \ f(0)=0, \ c>0,$ satisfies $f_c \in C[0,1] \cap BV[0,1]$. Figure \ref{fig2} shows the convergence of $\verb|HFD|(f_c,N,2)$ to the value $1$.
\end{example}

\begin{figure}
\centering
\includegraphics[scale=0.43]{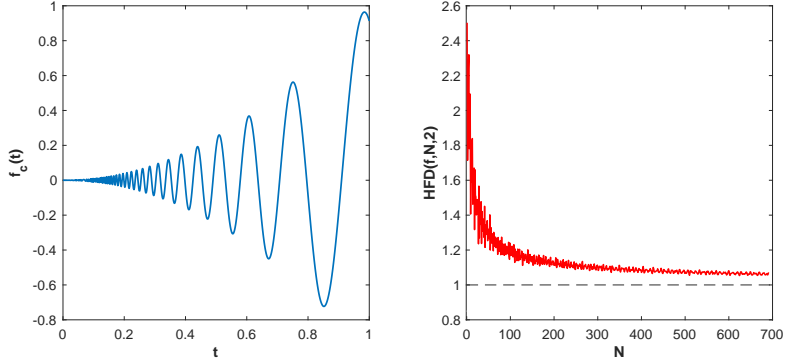}
\caption{Function $f_c(x) = x^2\sin(\frac{c}{x})$ with $c=20$ and the output of the Higuchi algorithm  for $N=2, \dots, 700$.}\label{fig2}
\end{figure}

\section{The Geometric Idea Behind the Higuchi Method}

Consider squares of the form
\begin{equation*}
S_{m_1,m_2} = [m_1\delta, (m_1+1)\delta] \times [m_2\delta, (m_2+1)\delta] \subset \R^2,
\end{equation*}
where $m_1,m_2 \in \Z$ and $\delta > 0$. We call a collection of such squares a \emph{$\delta$-mesh} of $\R^2$. For a nonempty bounded subset $F \subset \R^2$, we define the quantity $M_\delta(F)$ to represent the number of $\delta$-mesh squares that intersect $F$, i.e., 
\begin{equation}\label{mdel}
M_\delta(F) = | \{ (m_1,m_2) \in \Z^2 \ | \ S_{m_1,m_2} \cap F \neq \emptyset \} |.
\end{equation}
Using $\delta$-meshes, one obtains an equivalent definition for the box-counting dimension which turns out to be more convenient for a numerical approach.

\begin{theorem}
Let $F \subset \R^2$ be nonempty and bounded and let $M_\delta(F)$ be the number of intersecting $\delta$-mesh squares as defined in \ref{mdel}. Then
\begin{equation}\label{meshdim}
\mathrm{dim}_B(F) = \lim_{\delta \to 0+} \frac{\log M_\delta(F)}{\log \frac{1}{\delta}},
\end{equation}
provided this limit exists.
\end{theorem}
\begin{proof}
See, for instance, \cite{falconer1}, pp. 41--43.
\end{proof}

For $F \subset \R^2$ and $\delta > 0$ the area $\A_\delta$ of an intersecting $\delta$-mesh is given by
\begin{equation}\label{area}
\A_\delta = \delta^2M_\delta(F).
\end{equation}

Inserting \eqref{area} into \eqref{meshdim}, we can write the box-counting dimension of $F$ as
\begin{equation}\label{twominus}
\mathrm{dim}_B(F) = \lim_{\delta \to 0+} 2 - \frac{\log \A_\delta}{\log \delta}.
\end{equation}

At this point the question arises, what is a good approximation of the area $\A_\delta$ when $F$ is the graph of a bounded function $f : [0,1] \to \R$?

To this end, we start by taking numbers $N,k \in \N$ with $N \gg 1$ and $k \ll N$ such that $\frac{k}{N-1} \ll 1$. Set $\delta = \frac{k}{N-1}$. For subintervals $[x,x+\frac{k}{N-1}]$ of $[0,1]$ of length $\frac{k}{N-1}$ an approximation of an intersecting $\frac{k}{N-1}$-mesh of graph$(f|_{[x,x+\frac{k}{N-1}]})$ is given by
\begin{equation*}
\frac{k}{N-1} \left | f \left ( x+\frac{k}{N-1} \right ) - f(x) \right |.
\end{equation*}
Now for the whole graph of $f : [0,1] \to \R$, the sum
\begin{equation}\label{sum}
\frac{k}{N-1}\sum_{i=1}^{\left\lceil \frac{N-1}{k} \right\rceil} \left | f \left ( \frac{ik}{N-1} \right ) - f \left ( \frac{(i-1)k}{N-1} \right ) \right |
\end{equation}
approximates the area $\A_\delta$. This approach is shown in Figure \ref{fig3}. The bigger $N$ and the smaller $k$, the better the approximation.

\begin{figure}[h!]
\centering
\includegraphics[scale=0.5]{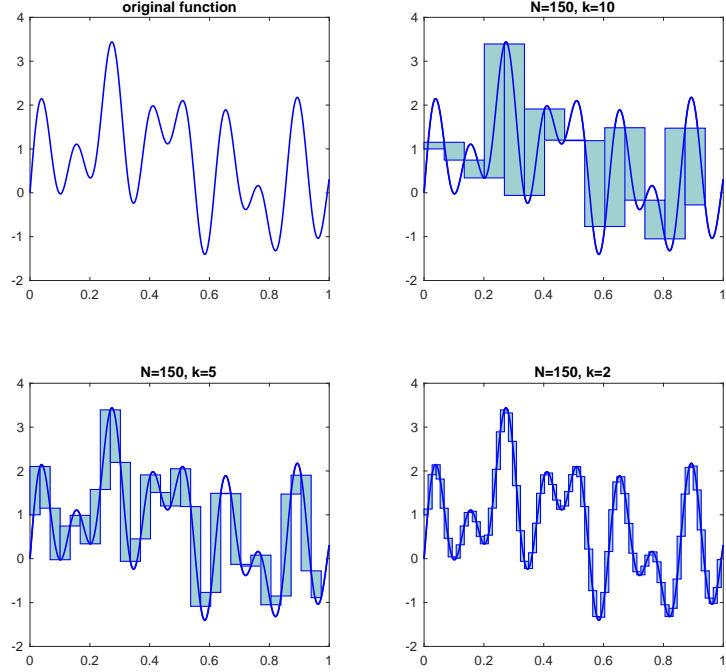}
\caption{Approximations of $\A_\delta$ with $k=2,5,10$ and $m=1$. The boxes represent the $\lceil \frac{N-m}{k} \rceil$ areas $\frac{k}{N-1}  \left| f ( \frac{ik}{N-1} ) - f \left ( \frac{(i-1)k}{N-1} \right )  \right|$ }\label{fig3}
\end{figure}

In \eqref{sum} we started on the very left of $[0,1]$ with $0$ and took $\lceil \frac{N-1}{k}\rceil$ steps to provide a value close to $\A_\delta$. We could also start at $\frac{1}{N-1}$ and approximate $\A_\delta$ by $\lceil\frac{N-2}{k} \rceil$ summands. It is reasonable to compute $k$ approximations of $\A_\delta$ with starting points
\begin{equation*}
0, \frac{1}{N-1}, \frac{2}{N-2}, \dots, \frac{k-1}{N-1}.
\end{equation*}
A starting point greater than $\frac{k-1}{N-1}$ would not be meaningful since otherwise a whole subinterval of length $\frac{k}{N-1}$ would be missing. Thus, we obtain $k$ values of the form
\begin{equation*}
\frac{k}{N-1}\sum_{i=1}^{\left \lceil \frac{N-m}{k} \right \rceil} \left | f \left ( \frac{m+ik-1}{N-1} \right ) - f \left ( \frac{m+(i-1)k-1}{N-1} \right ) \right |, \ \ m=1, \dots, k
\end{equation*}
which can be written as
\begin{equation*}
\frac{k}{N-1} V_{N,k,m}, \quad m=1, \dots, k,
\end{equation*}
where $V_{N,k,m}$ is defined exactly as in the Higuchi method. Note that the definition of $\frac{k}{N-1} V_{N,k,m}$ does not consider the graph of $f$ defined on the intervals 
\[
[0,{\frac{m-1}{N-1}}]\quad\text{and}\quad[{\frac{m+[\frac{N-m}{k}]k-1}{N-1}}, 1].
\]
Set $l := \frac{m-1}{N-1}$ and $r:= \frac{m+[\frac{N-m}{k}]k-1}{N-1}$. To obtain an even better approximation, we take these two subintervals into account and proceed in the following way.

The area of an intersecting $\delta$-mesh of $\mathrm{graph}(f|_{[l,r]})$ is approximately equal to $\frac{k}{N-1} V_{N,k,m}$. For a fractal function $f$, i.e., a function whose graphs exhibits fractal characteristics such as self-referentiality \cite{masso2}, the corresponding area of graph($f$) when defined on whole $[0,1]$ is approximately
\begin{equation*}
c \frac{k}{N-1} V_{N,k,m}
\end{equation*}
where $c$ is chosen in such a way that
\begin{equation*}
c \cdot \mathrm{length}([l,r]) = \mathrm{length}([0,1]) = 1,
\end{equation*}
where $\mathrm{length}(\cdot)$ denotes the length of an interval. The interval $[l,r]$ has length $\frac{\lceil\frac{N-m}{k}\rceil k}{N-1}$ which yields
\begin{equation*}
c = \frac{N-1}{\left \lceil \frac{N-m}{k} \right \rceil k}.
\end{equation*}
Hence, $c=C_{N,k,m}$, where $C_{N,k,m}$ is the normalization constant as defined in the Higuchi algorithm \cite{higuchi}. Finally, we obtain $k$ values 
\begin{equation*}
\frac{k}{N-1} C_{N,k,m} V_{N,k,m}, \quad m=1, \dots, k,
\end{equation*}
and their mean
\begin{equation*}
\widetilde{L}(\frac{k}{N-1}) \coloneqq \frac{1}{k} \sum_{m=1}^k \frac{k}{N-1} \,C_{N,k,m} V_{N,k,m} 
\end{equation*}
yields
\begin{equation*}
\A_{\frac{k}{N-1}} \approx \widetilde{L}\left(\frac{k}{N-1}\right).
\end{equation*}
Using \eqref{twominus} in combination with a linear regression line, we obtain
\begin{equation}\label{L}
\mathrm{dim}_B(\mathrm{graph}(f)) \approx 2 - L,
\end{equation}
where $L$ is the slope of a regression line through $k_\mathrm{max}$ data points
\begin{equation}\label{data}
\left\{\left( \log \frac{k}{N-1}, \log \widetilde{L}\left(\frac{k}{N-1}\right)\right) \,\Bigg\vert \,\, k = 1, \ldots, k_\mathrm{max}\right\}.
\end{equation}
This method yields a highly geometric approach for a numerical computation of the box-counting dimension. In fact, it is equivalent to the Higuchi method.
\begin{theorem}
Let $L$ be the slope of the regression line through the data \eqref{data}. Then $2-L = \verb|HFD| (f,N,k_{\mathrm{max}})$.
\end{theorem}
\begin{proof}

Using formula \eqref{slope} for the slope obtained by the method of least squares, we have
\begin{equation}\label{LL}
L = \frac{\sum\limits_{k=1}^{k_\mathrm{max}} (\log \frac{k}{N-1} - \overline{x}_*)(\log \tilde{L}(\frac{k}{N-1}) - \overline{y}_*)}{\sum\limits_{k=1}^{k_\mathrm{max}} (\log \frac{k}{N-1} - \overline{x}_*)^2}
\end{equation}
with
\begin{equation*}
\overline{x}_* = \frac{1}{k_\mathrm{max}} \sum_{k=1}^{k_\mathrm{max}} \log \frac{k}{N-1}, \ \ \overline{y}_* = \frac{1}{k_\mathrm{max}} \sum_{k=1}^{k_\mathrm{max}} \log \tilde{L}(\frac{k}{N-1}).
\end{equation*}
Define further
\begin{equation*}
\overline{x} = \frac{1}{k_\mathrm{max}} \sum_{k=1}^{k_\mathrm{max}} \log \frac{1}{k} \quad\text{and}\quad \overline{y} = \frac{1}{k_\mathrm{max}} \sum_{k=1}^{k_\mathrm{max}} \log L(k),
\end{equation*}
where $L(k)$ is defined as in the Higuchi algorithm,
\begin{equation*}
L(k) = \frac{1}{k} \sum_{m=1}^kL_m(k) = \frac{1}{k^2} \sum_{m=1}^k C_{N,k,m} V_{N,k,m}.
\end{equation*}
It follows that
\begin{equation}\label{A}
\log \frac{k}{N-1} - \overline{x}_* = - \left ( \log \frac{1}{k} - \overline{x} \right )
\end{equation}
as well as
\begin{equation}\label{B}
\widetilde{L}\left(\frac{k}{N-1}\right) = \frac{k^2}{N-1} \, L(k).
\end{equation}
A simple computation yields
\begin{equation}\label{C}
\log \widetilde{L}\left(\frac{k}{N-1}\right) - \overline{y}_* = -2(\log \frac{1}{k} - \overline{x}) + (\log L(k) - \overline{y}).
\end{equation}
Inserting \eqref{A} and \eqref{C} into \eqref{LL} gives
\begin{equation*}
L = 2-{\frac{\sum\limits_{k=1}^{k_\mathrm{max}} (\log \frac{1}{k} - \overline{x})(\log L(k) - \overline{y}))}{\sum\limits_{k=1}^{k_\mathrm{max}} (\log \frac{1}{k} - \overline{x})^2}}.
\end{equation*}
Setting $D:= \sum\limits_{k=1}^{k_\mathrm{max}} (\log \frac{1}{k} - \overline{x})^2$, we see that $D$ is exactly equal to the slope of a regression line through
\begin{equation*}
\mathcal{Z} = \left\{ \left(\log \frac{1}{k}, \log L(k)\right) \,\Bigg\vert\,\, k = 1, \ldots, k_\mathrm{max}\right\},
\end{equation*}
and therefore $D=\verb|HFD| (f,N,k_{\mathrm{max}})$. Consequently,
\begin{equation*}
2-L = \verb|HFD| (f,N,k_{\mathrm{max}})
\end{equation*}
which proves the assertion.
\end{proof}

\begin{example}[Weierstrass function]
Define $W:[0,1] \to \R$ by 
\begin{equation}\label{ws}
W(t) = \sum_{j=1}^\infty \lambda^{(s-2)j}\sin(\lambda^jt)
\end{equation}
where $\lambda > 1$ and $1<s<2$ are fixed. The function $W$ is continuous, nowhere differentiable, and it is known that $\mathrm{dim}_B(\mathrm{graph}(W)) = s$. (See, for instance, \cite{falconer1}, p. 162.) The sequence $\{\verb|HFD| (W,N,k_{\mathrm{max}})\}_{(N,k_\mathrm{max})}$ provides a sequence of approximations of $\mathrm{dim}_B(\mathrm{graph}(W))$. Figure \ref{fig4} represents the values of this sequence if we choose $k_\mathrm{max}=\lceil\frac{N}{2}\rceil$ and fixed $\lambda=5$ and $s=1.7$.
\end{example}

\begin{figure}[h!]
\centering
\includegraphics[scale=0.55]{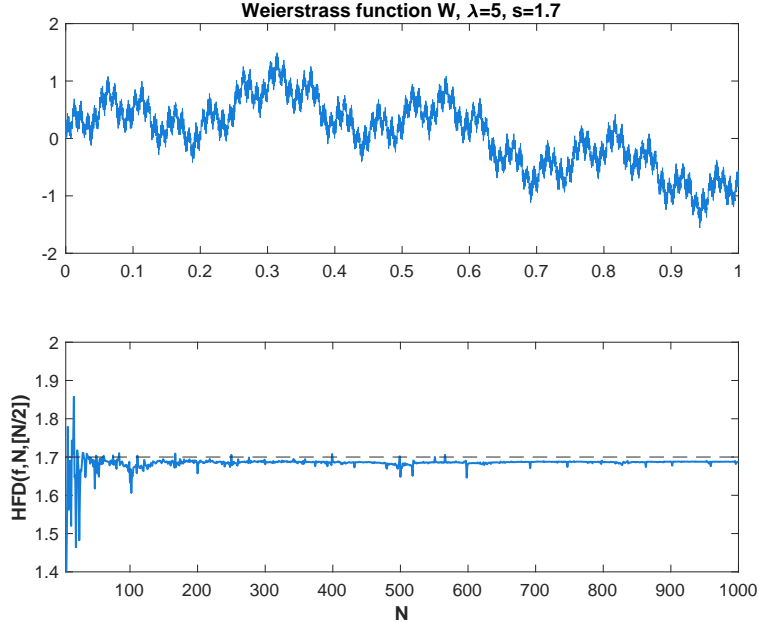}
\caption{Plot of the Weierstrass function $W$ with parameters $\lambda=5$ and $s=1.7$ together with the values $\texttt{HFD}(W,N,\lceil \frac{N}{2}\rceil)$, for $N=5, \dots, 1000$.}\label{fig4}
\end{figure}

\section{Stability Under Perturbations}

Let $f : [0,1] \to \R$ be a bounded function, $(N, k_\mathrm{max}) \in \N^2$ an admissible input, and $X_N : \{ 1, \dots, N \} \to \R$ the corresponding time series given by
\begin{equation*}
X_N(j) = \frac{j-1}{N-1}.
\end{equation*}
Let $L(1), \dots, L(k_\mathrm{max})$ be the lengths computed by the Higuchi algorithm. Assume that there is a $\kappa\in \{ 1, \dots, k_\mathrm{max} \}$ such that $L(\kappa) = 0$. This means $\kappa \not\in \mathcal{I}$ and the regression line through
\begin{equation*}
\mathcal{Z} = \{ ( \log \tfrac{1}{k}, \log L(k) ) \ | \ k \in \mathcal{I} \}
\end{equation*}
does not take the index $\kappa$ into account. Recall that $L(k)$ is defined by
\begin{equation*}
L(k) = \frac{1}{k} \sum _{m=1}^k L_m(k) = \frac{1}{k^2} \sum_{m=1}^k C_{N,k,m} V_{N,k,m}.
\end{equation*}
Thus, $L(\kappa) = 0$ means that
\begin{equation*}
V_{N,\kappa,m} = \sum_{i=1}^{\left \lceil \frac{N-m}{\kappa} \right \rceil} | X_N(m+i\kappa) - X_N(m+(i-1)\kappa)| = 0, \quad \forall m = 1, \dots, \kappa.
\end{equation*}
This implies that for every $m$ the values 
\begin{equation*}
X_N(m+(i-1)\kappa), \ \ \ i=1, \dots, \left \lceil \frac{N-m}{\kappa} \right \rceil,
\end{equation*}
lie on a line. Suppose we perturb a value $X_N(j)$ by a small constant $0<\varepsilon \ll 1$. For simplicity, we assume $j=1$. Define a new time series $X_N^\varepsilon$ via
\begin{equation*}
X_N^\varepsilon(j) \coloneqq
\begin{cases}
X_N(1) + \varepsilon, & j=1;\\
X_N(j), & \text{else}.\\
\end{cases}
\end{equation*}
Denote by $L^\varepsilon(k)$, $L_m^\varepsilon(k)$, $V_{N,k,m}^\varepsilon$, $\mathcal{I}^\varepsilon$, $\mathcal{Z}^\varepsilon$, and $D^\varepsilon$ the corresponding values in the Higuchi algorithm with respect to the perturbed time series. Observe that the values
\begin{equation}\label{coll}
X_N^\varepsilon(m+(i-1)\kappa), \quad i=1, \dots, \left \lceil \frac{N-m}{\kappa} \right \rceil,
\end{equation}
lie on a line if any only if $m > 1$. For $m=1$ the perturbation by $\varepsilon$ destroys the collinearity. If follows that $V_{N,\kappa,m}^\varepsilon = 0$ for $m > 1$ and $V_{N,\kappa,m}^\varepsilon = \varepsilon$. Hence,
\begin{equation*}
L^\varepsilon(\kappa) = \tfrac{1}{\kappa}\, L_1^\varepsilon(\kappa) = \tfrac{1}{\kappa^2} C_{N,\kappa,1} \varepsilon
\end{equation*}
which implies that $0<L^\varepsilon(\kappa) \ll 1$. On the other hand, for all $k$ which were originally in the index set $\mathcal{I}$, i.e., for which $L(k) \neq 0$, we have
\begin{equation*}
L^\varepsilon(k) = \frac{1}{k} \sum_{m=1}^k L_m^\varepsilon(k) = \frac{1}{k} \left ( L_1^\varepsilon(k) + \sum_{m=2}^k L_m(k) \right )
\end{equation*}
with
\begin{align*}
L_1^\varepsilon(k) &= \frac{1}{k}C_{N,k,1}\, \cdot\\
& \left ( | X_N(1+k)-(X_N(1)+\varepsilon)| + \sum_{i=2}^{\left \lceil \frac{N-1}{\tilde{k}} \right \rceil} | X_N(m+ik) - X_N(m+(i-1)k)) | \right ).
\end{align*}

It follows that for $\varepsilon$ sufficiently small, $k \in \mathcal{I}^\varepsilon$ and $L^\varepsilon(k) \approx L(k)$.
Therefore, the new data set $\mathcal{Z}^\varepsilon$ consists of all points $(\log \frac{1}{k}, \log L^\varepsilon(k))$ with $k \in \mathcal{I}$ and in addition of at least one new point which is 
\begin{equation*}
P \coloneqq \left ( \log \frac{1}{\kappa}, \log L^\varepsilon(\kappa) \right ).
\end{equation*}
Whilst the lengths $L^\varepsilon(k)$, $k \in \mathcal{I}$, stay nearly untouched, the new point $P$ can have a significant influence on the slope $D^\varepsilon$ of the new regression line. This follows from the fact that
\begin{equation*}
\lim_{\varepsilon \to 0+} \log L^\varepsilon(k) =
\begin{cases}
\log L(k), & k \in \mathcal{I};\\
-\infty, & k = \kappa.\\
\end{cases}
\end{equation*}

\begin{example}\label{ex5.1}
Let $(N,k_\mathrm{max}) \in \N^2$ be admissible, $\kappa \in \{ 1, \dots, k_\mathrm{max} \}$, and $c \coloneqq (c_1, \dots, c_{\kappa}) \in \R^{\kappa}$. Define $f_c:[0,1] \to \R$ to be a continuous interpolation function such that for all $m = 1, \dots, \kappa$,
\begin{equation*}
X_N(m+(i-1)\kappa) = c_m \, , \quad \forall i = 1, \dots, \left \lceil \frac{N-m}{\kappa} \right \rceil.
\end{equation*}
It follows form the preceding discussion that $L(\kappa) = 0.$ Figure \ref{fig5} shows an example with $\kappa = 10$, $N=150$, and $k_\mathrm{max} = 30$:
\begin{equation*}
c \coloneqq (1, 1.1, 1.3, 1.4, 1.3, 1.4, 1.3, 1.4, 1.3, 1.1),
\end{equation*}
where $f_c$ a linear spline through the interpolation points. Note that the slope $D$ of the regression line of the perturbed time series is approximately $D=3.5$, a nonsensical result as the box-counting dimension of $\graph (f_c)$ must be $\leq 2$.
\end{example}

\begin{figure}[h!]
\centering
\includegraphics[scale=0.5]{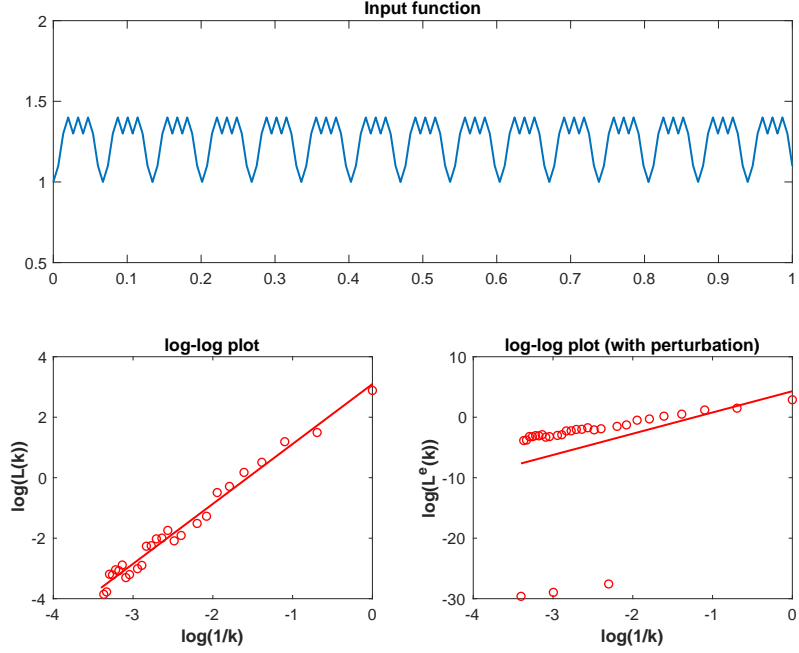}
\caption{Plot of the interpolation function $f_c$ from Example \ref{ex5.1} and the corresponding $\log-\log$ plots with and without perturbation. The chosen constants are $N=150$, $k_\mathrm{max}=30$ and the time series is perturbed by $\varepsilon = 10^{-10}$. The slope of the first regression line is $\approx 1.9$ and the slope of the second is $\approx 3.5$.}\label{fig5}
\end{figure}

\begin{example}\label{evenodd}
Let $(N,k_\mathrm{max}) \in \N^2$ be admissible and $c_1, c_2 \in \R$ with $c_1 \neq c_2.$ Consider a function $f : [0,1] \to \R$ with the property that
\begin{equation*}
f \left ( \frac{j-1}{N-1} \right ) =
\begin{cases}
c_1, & j \ \textrm{odd};\\
c_2, & j \ \textrm{even}.\\
\end{cases}
\end{equation*}
Assume that $k$ is odd, $1 \leq m \leq k$ and $i \in \Z$. Then $m+ik$ is even if and only if $m+(i-1)k$ is odd. Therefore,
\begin{equation*}
V_{N,k,m} = \sum_{i=1}^{\left \lceil \frac{N-m}{k} \right \rceil} |X(m+ik)-X(m+(i-1)k)| = |c_1-c_2| \left \lceil \frac{N-m}{k} \right \rceil.
\end{equation*}
This implies 
\begin{equation*}
L(k) = \frac{(N-1)|c_1 - c_2|}{k^2}.
\end{equation*}

For even $k$, $1 \leq m \leq k$ and $i \in \mathbb{Z}$ the values $m+ik$ are either even (for $m$ even) or odd (for $m$ odd). Hence, $0=V_{N,k,m} = L(k)$. Hence,
\begin{equation*}
\mathcal{I} = \{ 1 \leq k \leq k_\mathrm{max} \ | \ k \ \text{odd} \}
\end{equation*}
and consequently
\begin{equation*}
L(k) \propto k^{-2}, \ \ k \in \mathcal{I}.
\end{equation*}
By Proposition \ref{basics} we obtain $\verb|HFD| (f,N,k_{\mathrm{max}}) = 2.$
\end{example}

\begin{corollary}
For every admissible input $(N,k_\mathrm{max}) \in \N^2$ there exists a bounded function $f : [0,1] \to \R$ with
\begin{equation*}
\verb|HFD| (f,N,k_{\mathrm{max}}) = 2.
\end{equation*}
Further $f$ can be chosen to be of bounded variation. In this case:
\begin{equation*}
\mathrm{dim}_B(\mathrm{graph}(f)) = 1 \neq 2 = \verb|HFD| (f,N,k_{\mathrm{max}}).
\end{equation*}
\end{corollary}
\begin{proof}
Following the construction of the preceding example, we can choose $f$ to be an interpolation function through a finite set of interpolation points. In particular, we can choose $f$ as a function of bounded variation such as a linear spline. Then
\begin{equation*}
\mathrm{dim}_B(\mathrm{graph}(f)) = 1
\end{equation*}
but $\verb|HFD| (f,N,k_{\mathrm{max}}) = 2$.
\end{proof}

We continue with the function from Example \ref{evenodd}. A perturbation of $X_N(1)$ by a small value $\varepsilon > 0$ yields
\begin{equation*}
L(k) \neq 0, \quad \forall k = 1, \dots, k_\mathrm{max}.
\end{equation*}
So the index set $\mathcal{I}$ which previously consisted  of all odd values of $k$ now changes to $\mathcal{I} = \{ 1, \dots, k_\mathrm{max} \}$. Hence, $\mathcal{I}$ is maximal. Furthermore,
\begin{equation*}
\lim_{\varepsilon \to 0+} \log L^\varepsilon(k) = -\infty,
\end{equation*}
for all even $k$. Figure \ref{fig6} displays the corresponding regression lines.

\begin{figure}[h!]
\centering
\includegraphics[scale=0.5]{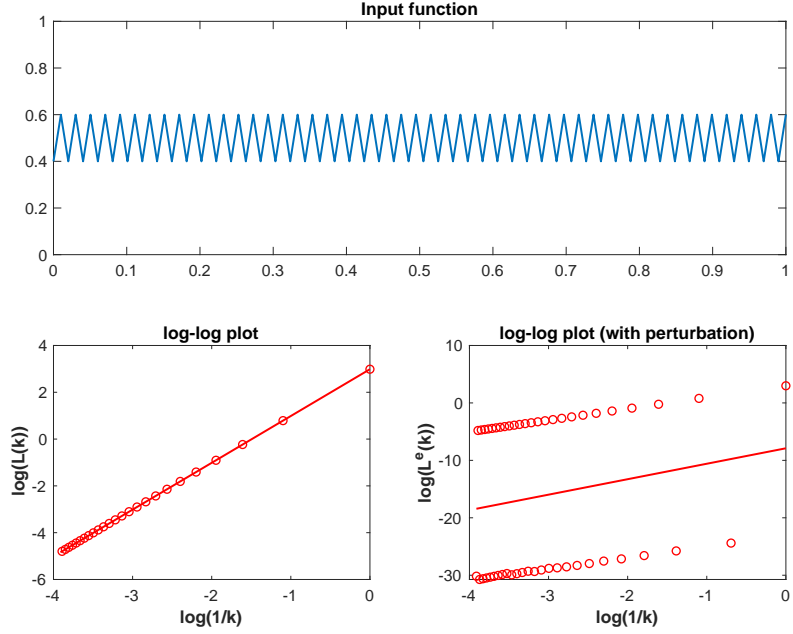}
\caption{Plot of the interpolation function $f$ of Example \ref{evenodd} with $N=100$, $k_\mathrm{max} = \lceil\frac{N}{2}\rceil$, $c_1 = 0.4$, and $c_2 = 0.6.$ The bottom left plot shows that the points in $\mathcal{Z}$ lie exactly on a line with slope 2. The bottom right plot displays the data set $\mathcal{Z}^\varepsilon$ with $\varepsilon=10^{-10}$. The regression line has slope $D^\varepsilon \approx 2.7$, in contradiction to the fact that the box-counting dimension of $\graph (f)$ must be $\leq 2$}\label{fig6}
\end{figure}

\section*{Appendix: Proof of Theorem \ref{thm3.3}}\label{app}

\noindent
\textbf{Theorem 3.3.}\textit{
Let $f \in C[a,b] \cap BV[a,b]$ and let $(P_n)_{n \in \N}$ be a sequence of partitions of $[a,b]$ such that $|P_n| \to 0$ as $n \to \infty$. Then
\begin{equation*}
\lim_{n \to \infty} V_{P_n}(f) = V_a^b(f).
\end{equation*}
}
\vspace*{-20pt}
\begin{proof}
Let $\varepsilon > 0$ and let $P \coloneqq \{ 0 = x_0 < \cdots < x_n = 1 \}$ be a partition of $[0,1]$ such that
\begin{equation*}
V_0^1(f) - \varepsilon \leq V_P(f) \leq V_0^1(f).
\end{equation*}
Let $I_k \coloneqq (x_{k-1}, x_k)$, $k = 1, \dots, n.$ The uniform continuity of $f$ on the compact set $[0,1]$ yields the existence of a $\delta > 0$ such that
\begin{equation}\label{unif}
|x-y| < \delta \implies |f(x)-f(y)| < \frac{\varepsilon}{2n} \ \ \forall x,y \in [0,1].
\end{equation}
Furthermore, $\delta$ can be chosen in such a way that any partition of $Q$ with $|Q| < \delta$ contains at least two points in every subinterval $I_k$. Let $Q$ be such a partition. Define
\begin{equation*}
Q_k \coloneqq Q \cap I_k, \quad k=1, \dots, n
\end{equation*}
then $Q_k$ contains at least two points. Set $a_k \coloneqq \min Q_k$ and $b_k \coloneqq \max Q_k$.
Let $R$ be the partition $R = P \cup Q_1 \cup \cdots \cup Q_n.$ Since $R$ is a refinement of $P$, we have
\begin{equation*}
V_R(f) \geq V_P(f).
\end{equation*}
By the definition of $R$ and inequality \eqref{unif}, it follows that
\begin{align*}
V_R(f) & = \sum_{k=1}^n V_{Q_k}(f) + \sum_{k=1}^n (|f(x_{k-1})-f(a_k)| + |f(x_k)-f(b_k)|) \\
 & \leq \sum_{k=1}^n V_{Q_k}(f) + \sum_{k=1}^n \left( \frac{\varepsilon}{2n} + \frac{\varepsilon}{2n} \right) \leq V_Q(f) + \varepsilon.
\end{align*}
Hence,
\begin{equation*}
V_0^1(f) - \varepsilon \leq V_P(f) \leq V_Q(f) + \varepsilon
\end{equation*}
and the result follows.
\end{proof}

\bibliographystyle{plain}
\bibliography{Higuchi}

\end{document}